\documentclass[11pt,a4paper]{amsart}
\usepackage[top=4cm,bottom=4cm,left=3.5cm,right=3.5cm]{geometry}

\usepackage{amsmath,amssymb,amsthm}
\usepackage[english]{babel}
\usepackage{mathtools}
\usepackage{enumerate}
\usepackage[numbers]{natbib}
\usepackage[hidelinks]{hyperref}
\usepackage{tikz}
\usepackage{tkz-graph}

\usepackage{xcolor}

\DeclarePairedDelimiter{\abs}{\lvert}{\rvert}

\DeclareMathOperator{\disc}{disc}

\newcommand{\vF}{\mathbb{F}}

\newcommand{\cS}{\mathcal{S}}

\newcommand{\cL}{\mathcal{L}}
\newcommand{\cI}{\mathcal{I}}
\newcommand{\cO}{\mathcal{O}}

\newcommand{\Ansq}{\mathcal{N}}
\newcommand{\Lnsq}{\mathcal{L}}
\newcommand{\Airr}{\mathcal{M}}
\newcommand{\Lirr}{\mathcal{I}}

\newcommand{\init}{\mathfrak{I}}
\newcommand{\dist}[1]{{[#1]}}
\newcommand{\reg}[1]{{\{#1\}}}

\newtheorem{theorem}{Theorem}[section]
\newtheorem{lemma}[theorem]{Lemma}
\newtheorem{proposition}[theorem]{Proposition}
\newtheorem{corollary}[theorem]{Corollary}
\theoremstyle{definition}

\newtheorem{definition}[theorem]{Definition}
\theoremstyle{remark}
\newtheorem{remark}[theorem]{Remark}
\newtheorem{example}[theorem]{Example}

\numberwithin{equation}{section}

\begin{document}

\title[Irreducible compositions of degree two polynomials]{Irreducible compositions of degree two polynomials over finite fields have regular structure}

\author[A. Ferraguti]{Andrea Ferraguti}
\address{University of Cambridge\\
DPMMS\\
Centre for Mathematical Sciences\\
Wilberforce Road, Cambridge, CB3 0WB, UK\\
}
\email{af612@cam.ac.uk}

\author[G. Micheli]{Giacomo Micheli}
\address{Mathematical Institute\\
University of Oxford\\
Woodstock Rd\\ 
Oxford OX2 6GG, United Kingdom
}
\email{giacomo.micheli@maths.ox.ac.uk}

\author[R. Schnyder]{Reto Schnyder}
\address{Institute of Mathematics\\
University of Zurich\\
Winterthurerstrasse 190\\
8057 Zurich, Switzerland\\
}
\email{reto.schnyder@math.uzh.ch}

\thanks{The first author was supported by Swiss National Science Foundation grant number 168459. The Second Author is thankful to Swiss National Science Foundation grant number 161757 and 171248.}


\begin{abstract}
Let $q$ be an odd prime power and $D$ be the set of monic irreducible polynomials in $\vF_q[x]$ which can be written as a composition of monic degree two polynomials. In this paper we prove that $D$ has a natural regular structure by showing that there exists a finite automaton having $D$ as accepted language. Our method is constructive.
\end{abstract}

\maketitle

\section{Introduction}
\label{sec:introduction}
It has been of great interest in recent years the study of irreducible
polynomials which can be written as composition of degree two polynomials (see
for example
\citep{bib:ahmadi2009note,
bib:ahmadi2012stable,
ferraguti2016sets,
bib:gomez2010estimate,
bib:gomez2017complexity,
bib:heathbrownmicheli,
jones2007iterated,
jones2008density,
bib:jones2012iterative,
bib:jones2012settled,
bib:ostafe2010length}).
Such polynomials are also used in other contexts, see for example Rafe Jones' construction
of irreducible polynomials reducible modulo every prime
\cite{bib:jones2012iterative} or the proof of \cite[Conjecture
1.2]{bib:andrade2013special} in \cite{bib:ferraguti2016existence}. In this
paper we explain how the theory of this class of polynomials completely fits in
a general context which allows the use of powerful machinery coming from the
 theory of finite automata (in characteristic different from $2$).
In fact, we show that some irreducibility questions over finite fields can be
translated into automata theoretical ones (see
Definition~\ref{def:automaton_nonsquare} and Theorem~\ref{thm:irred_regular}).
As a side result, we also obtain that the set of irreducible polynomials which
can be written as the composition of degree two polynomials is naturally endowed
with a regular structure given by Theorem~\ref{thm:disjoint}.

Let $\vF_q$ be a finite field of odd characteristic and let $\cS \subset \vF_q[x]$ be a set of monic degree two polynomials. 
In this paper, we consider the set $\cS$ to be an alphabet, and a word $f_1
\cdots f_k \in \cS^*$ corresponds to the composition $f_1 \circ \cdots \circ
f_k \in \vF_q[x]$.
The empty word naturally corresponds to $x$.
Let $\Lirr \subset \cS^*$ be the language of words whose corresponding
compositions are irreducible.
Our goal is to show that $\Lirr$ is a regular language by  providing an automaton that accepts it.
The entire theory lifts to local fields under the assumption that the set $\cS$ is finite and none of its elements has discriminant in the maximal ideal of the local field.

\section{Distinguished sets and freedom}
We include in this section some elementary facts concerning the freedom of the monoid generated by a finite set of irreducible degree two polynomials. These results will be needed in Section~\ref{sec:automaton_irreducible}.
Each polynomial $f \in \cS$ can be uniquely written as $f = (x - a_f)^2 - b_f$
for some $a_f,b_f\in \vF_q$.
We define $D_\cS = \{b_f : f \in \cS\}$ to be the \emph{distinguished set} of
$\cS$.

We denote by $\cS^*$ the set of words over the alphabet $\cS$, so $\cS^*$ is
the free monoid generated by the symbols in $\cS$. Let $C_{\cS}\subseteq \vF_q[x]$ be the
compositional monoid generated by $\cS$, that is the smallest subset of
$\vF_q[x]$ containing $\cS$ and $x$ which is closed by composition.

We will denote by $\pi$ the natural surjective morphism of monoids $\cS^*\longrightarrow C_{\cS}$ which maps a word $f_1f_2\ldots f_k\in\cS^*$ to the composition $f_1\circ f_2\circ\ldots\circ f_k\in\vF_q[x]$.
Notice that, for an element $f\in \cS$ we will denote by $f^{(n)}$ the $n$-fold composition of $f$ with itself.

For $b\in D_\cS$, we define $A_b$ as the subset of all $a$ in $\vF_q$ such that
there exists $f\in \cS$ with $f=(x-a)^2-b$.
For any of the $A_b$, we define the difference set
\[A_b-A_b=\{a-a': a,a'\in A_b\}.\]
We can define a relation $\sim$ on $\cS^*$ by setting $u\sim w$ if there exists $\ell\in \bigcup_{b\in D_\cS} \left(A_b-A_b\right)$ for which 
$\pi(u)+\ell=\pi(w)$. This relation is symmetric and reflexive but not transitive, unless  $\bigcup_{b\in D_\cS} \left(A_b-A_{b}\right)$ is an additive subgroup of $\vF_q$.

In this section we provide a computable condition to establish whether $C_\cS$
is a free monoid, which will be needed later on. 
\begin{proposition}\label{thm:relation}
Let $u,v$ be words of $\cS^*$ of equal length $n \ge 1$. Let $u',v'$ be the $(n-1)$-suffixes of $u$ and $v$ respectively. Then
\begin{enumerate}
\item[(i)] $\pi(u)=\pi(v)$ implies  $u'\sim v'$
\item[(ii)] $u\sim v$ if and only if  $\pi(fu)=\pi(gv)$ for some $f,g\in \cS$.
\end{enumerate}
\end{proposition}
\begin{proof}
Let us first prove (i). Suppose $\pi(u)=\pi(v)$ and let us write
\[\pi(u)=(h_1-a)^2-b=(h_2-a')^2-b'=\pi(v)\]
for $h_1 = \pi(u')$, $h_2 = \pi(v')$.
Then we have $(h_1-a-h_2+a')(h_1+h_2-a-a')=b-b'$. Since $h_1+h_2-a-a'$ has positive degree, this forces $b=b'$ and $h_1-a-h_2+a'=0$. Now it is clear that $a,a'\in A_b$, which implies $a'-a\in A_b-A_b$, and then $u'\sim v'$.

Let us now prove (ii). If $u\sim v$, by definition we have $\pi(u)-a=\pi(v)-a'$
for $a-a'\in A_b-A_b$ for some $b$. Now, by squaring and subtracting $b$ on both
sides of the equality we get $f(\pi(u))=g(\pi(v))$ for some $f,g\in \cS$, and
hence $\pi(fu) = \pi(gv)$.
Vice versa, if there exists $f,g\in \cS$ such that $\pi(f u)=\pi(g v)$, then (i)
applies.
\end{proof}

\begin{lemma}\label{thm:relation_max_min}
	Let $u,v$ be words of $\cS^*$ of equal length $n \ge 1$.
	If $\abs{D_{\cS}}=\abs{\cS}$ or $\abs{D_\cS}=1$, then we have that
	$u \sim v$ if and only if $\pi(u) = \pi(v)$.
\end{lemma}
\begin{proof}
One direction is trivial: If $\pi(u)=\pi(v)$ then $u\sim v$.
For the other direction, we look at the two cases separately.

In the case $\abs{D_{\cS}}=\abs{\cS}$, it follows from $u \sim v$ that $\pi(u) -
\pi(v) = c \in A_b - A_b$ for some $b \in D_\cS$.
However, $A_b$ consists of only one element, so $c = 0$.

For the case $\abs{D_{\cS}}=1$, assume that $u \sim v$, so $\pi(u)=\pi(v) + c$
for some $c\in \vF_q$.
Let $u',v'$ be the $(n-1)$-suffixes of $u$ and $v$.
Then, since $\abs{D_\cS}=1$, we have that $(\pi(u')-a_1)^2-(\pi(v')-a_2)^2=c$
for some $a_1,a_2\in\vF_q$.
As $c$ is constant, this forces $(\pi(u')-a_1)-(\pi(v')-a_2)=0$, which in turn
forces $c=0$ and hence $\pi(u) = \pi(v)$.
\end{proof}

The following proposition shows that the freedom of the monoid is ensured whenever $D_\cS$ is either maximal or minimal. 

\begin{proposition}\label{thm:freedom_max_min}
If $\abs{D_{\cS}}=\abs{\cS}$ or $\abs{D_\cS}=1$, then $C_{\cS}\cong \cS^*$.
\end{proposition}
\begin{proof}
Clearly, a polynomial of degree two in $C_{\cS}$ cannot have two distinct writings in terms of compositions.
Let $F$ be a polynomial in $C_\cS$ of minimal degree with two different writings,
i.e.\@ such that $F=\pi(fu)=\pi(gv)$ for $f,g\in \cS$ and $u, v\in \cS^*$ of
positive length.
From $\pi(fu)=\pi(gv)$ one deduces by Proposition~\ref{thm:relation} that $u\sim v$.
Lemma~\ref{thm:relation_max_min} now gives $\pi(u)=\pi(v)$, which implies  $u=v$
by the minimality of $F$.
\end{proof}

\begin{corollary}\label{thm:freedom2}
If $\abs{\cS}=2$ then $C_{\cS}\cong \cS^*$.
\end{corollary}
\begin{proof}
Immediate by observing that $\abs{D_\cS}=1$ or $\abs{D_\cS}=\abs{\cS}=2$.
\end{proof}

\section{An automaton for irreducible compositions}
\label{sec:automaton_irreducible}

\subsection{Capelli's Lemma}
\label{sec:capellis_lemma}

In this subsection we describe the basic tools needed to establish the main result. We start with a well known result by Capelli, which gives a necessary and sufficient criterion to control the irreducibility of the composition of two polynomials.
\begin{lemma}[Capelli's Lemma]\label{thm:capelli}
	Let $K$ be a field and $f,g \in K[x]$ polynomials.
	Let $\beta \in \overline{K}$ be a root of $g$.
	Then, $g \circ f$ is irreducible over $K$ if and only if $g$ is irreducible
	over $K$ and $f - \beta$ is irreducible over $K(\beta)$.
\end{lemma}

We now use Capelli's Lemma to produce a simple ancillary result which will help us in what follows.

\begin{lemma}\label{thm:irred_short}
	Let $g \in \vF_q[x]$ be monic and irreducible of even degree, and let $f = (x -
	a_f)^2 - b_f \in \vF_q[x]$.
	Then, $g \circ f$ is irreducible if and only if $g(-b_f)$ is not a square in
	$\vF_q$.
\end{lemma}
\begin{proof}
	Let $d = \deg(g)$, and let $\beta \in \vF_{q^d}$ be a root of $g$.
	According to Lemma~\ref{thm:capelli}, $g \circ f$ is irreducible over $\vF_q$
	if and only if $f - \beta$ is irreducible over $\vF_{q^d}$.
	Writing $f-\beta = (x - a_f)^2 - (b_f + \beta)$, this is equivalent to $b_f + \beta$
	not being a square in $\vF_{q^d}$.
	Let $N\colon \vF_{q^d} \to \vF_q$ be the norm map.
	If $\beta_1, \ldots, \beta_d$ are the roots of $g$, we have
	\begin{equation*}
		N(b_f + \beta) = \prod_{i=1}^{d}(b_f + \beta_i) = (-1)^d \prod_{i=1}^{d}
		((-b_f) - \beta_i) = g(-b_f),
	\end{equation*}
	since $d$ is even.
	Now we can conclude, 
	since $b_f + \beta$ is a nonsquare in $\vF_{q^d}$ if and
	only if $N(b_f + \beta) = g(-b_f)$ is a nonsquare in $\vF_q$.
\end{proof}

We are now ready to state one of the basic ingredients of the proof of the main theorem, which will allow us to ``push'' irreducibility questions for compositions of degree two polynomials on a finite level. 

\begin{proposition}\label{thm:irred_long}
	Let $f_1, \ldots, f_k \in \vF_q[x]$ be monic polynomials of degree two.
	Write $f_i = (x - a_i)^2 - b_i$ for all $i$.
	Then, $f_1 \circ \cdots \circ f_k$ is irreducible if and only if all of the
	following are nonsquares in $\vF_q$:
	\begin{itemize}
		\item $b_1$
		\item $f_1(-b_2)$
		\item[] $\vdots$
		\item $(f_1 \circ \cdots \circ f_{k-1})(-b_k)$.
	\end{itemize}
\end{proposition}
\begin{proof}
	Clearly, $f_1$ is irreducible if and only if $b_1$ is a nonsquare.
	The rest follows by inductive application of Lemma~\ref{thm:irred_short}.
\end{proof}

\subsection{Brief interlude on Automata Theory}

In this subsection we recall the basic results needed in the next section. 
For the definition of \emph{deterministic finite automaton} (DFA),
and \emph{nondeterministic finite automaton} (NFA) we refer for example to
\cite[Chapter~2]{bib:hopcroftintroduction}.
Since all the automata in the paper will have a finite set of states, we will often omit the word finite.

Let $\Sigma$ be a set of symbols (an \emph{alphabet}) and $\Sigma^*$ be the set
of words over $\Sigma$, that is the free monoid generated by it.
Let us recall that a subset $\cL$ of $\cS^*$ is called a \emph{language}. Let $\cdot$ be 
the usual binary operation in $\Sigma^*$ (i.e. concatenation of words) and ${}^*$ be the unary operation on languages defined by $\cL\mapsto \cL^*$ where $\cL^*$ is the smallest submonoid of $\Sigma^*$ containing $\cL$ (in the context of languages, this operation is often called \emph{Kleene star}).

 A language is said to be
\emph{regular} if it is finite or can be expressed recursively starting from
finite sets using the operations $\cup,\cdot,{}^*$ (see \cite{bib:hopcroftintroduction} for more details).  The following fact is
well known.

\begin{theorem}
A language is regular if and only if it is accepted by a DFA.
\end{theorem}

We will need the following fundamental results from the theory of Automata.

\begin{theorem}
If a language $\mathcal{L}$ is accepted by a DFA or an NFA, then it is regular. 
\end{theorem}
Roughly, the theorem above states that the accepted languages of NFAs
do not generalize the notion of regular language.

We will also be using the notion of a \emph{partial deterministic finite
automaton}, which is the same as a DFA, except the transition function is
actually a partial function.
If, when reading a word, a required transition is not defined, the word is
rejected.
Clearly, a partial DFA is a special case of an NFA, so languages accepted by
partial DFAs are also regular.

\subsection{Putting all together: building the automaton}
We first define a finite deterministic automaton $\Ansq=\Ansq(\cS)$ using the data contained in $\cS$. 

\begin{definition}\label{def:automaton_nonsquare}
The states of the automaton $\Ansq(\cS)$ are given by the following:
	\begin{itemize}
		\item A special start state $\init$.
			It is accepting.
		\item For each $a \in \vF_q$, we have a distinguished state
			$\dist{a}$.
			It is accepting if $-a$ is a nonsquare.
		\item For each $a \in \vF_q$, we have a state $\reg{a}$.
			It is accepting if $a$ is a nonsquare.
	\end{itemize}
	The transitions are as follows:
	\begin{itemize}
		\item For each $f \in \cS$, we have a transition
			$\init \xrightarrow{f} \dist{-b_f}$.
		\item For each $f \in \cS$ and each $a \in \vF_q$, we have a transition
			$\dist{a} \xrightarrow{f} \reg{f(a)}$.
		\item For each $f \in \cS$ and each $a \in \vF_q$, we have a transition
			$\reg{a} \xrightarrow{f} \reg{f(a)}$.
	\end{itemize}
\end{definition}

\begin{remark}
	The reason we distinguish between the states $\reg{a}$ and $\dist{a}$ is that
	they may be accepting at different times:
	$\reg{a}$ accepts if $a$ is nonsquare, $\dist{a}$ if $-a$ is nonsquare.
	In the case that $-1$ is a square in $\vF_q$, the two are equivalent and we
	can identify the two types of states.
\end{remark}

\begin{theorem}\label{thm:irred_regular}
	The language $\Lirr$ of irreducible compositions is regular.
\end{theorem}
\begin{proof}
	Let $\Lnsq$ be the regular language over the alphabet $\cS$ that is accepted
	by the automaton $\Ansq$ reading from \emph{right to left}.
	It is easy to see that a single letter $f$ is in $\Lnsq$ if and only if $b_f$
	is nonsquare.
	Furthermore, a word $f_1 \cdots f_k$, $k \ge 2$, is in $\Lnsq$ if and only if
	$(f_1 \circ \ldots \circ f_{k-1})(-b_{f_k})$ is nonsquare.
	By Proposition~\ref{thm:irred_long}, it follows that the word $f_1 \cdots
	f_k$ corresponds to an irreducible polynomial if and only if each prefix $f_1
	\cdots f_l$, $l \le k$, lies in $\Lnsq$.
	In other words, $\Lirr$ is the language of all words whose every prefix is in
	$\Lnsq$.

	We want to prove that $\Lirr$ is regular. To do so, let us first construct a
	non-deterministic automaton $\mathcal U$ from $\Ansq$ by reversing all the
	arrows of $\Ansq$ and setting the start state of $\Ansq$ as final state of
	$\mathcal U$ and the final states of $\Ansq$ as start states of $\mathcal U$.
	Observe that the accepted language of $\mathcal U$ is again the language
	$\Lnsq$, this time reading from left to right as usual. Now we use subset construction (see
	e.g.~\cite[Section~2.3.5]{bib:hopcroftintroduction}) to generate a new
	deterministic automaton $\mathcal V$ having the same accepted language as
	$\mathcal U$.
	Finally, we remove all non-accepting states from $\mathcal V$ to obtain the
	partial automaton $\Airr$.
	It is easily seen that $\Airr$ accepts exactly those words whose every prefix
	is accepted by $\mathcal V$, which as explained above is $\Lirr$.
\end{proof}

\begin{remark}\label{rem:prefixclosed}
Notice that the language accepted by the final automaton $\mathcal M$ is prefix closed.
\end{remark}

We now provide an example to see the theorem in action.

\begin{example}\label{ex:1}
	As a simple example, consider the case $q = 5$ and $\cS = \{f, g\}$ with
	$f = x^2 - 2$ and $g = (x-1)^2 - 3$.

	We first construct the interim automaton $\Ansq$ using the method described in
	Definition~\ref{def:automaton_nonsquare}.
	Since $p \equiv 1 \mod 4$, we can identify the nodes $\dist{a}$ and $\reg{a}$.
	Note that we have removed the node $\reg{0}$ since it is not reachable from
	$\init$.
	The result is seen in Figure~\ref{fig:ex1_Ansq}.

	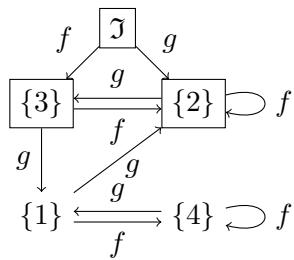
\begin{figure}[h]
		\centering
		\begin{tikzpicture}[auto]
			\node (I) at (0,0) [draw] {$\init$};

			\node (a2) at ( 1,-1) [draw] {$\reg{2}$};
			\node (a3) at (-1,-1) [draw] {$\reg{3}$};
			\node (a1) at (-1,-2.5) [] {$\reg{1}$};
			\node (a4) at ( 1,-2.5) [] {$\reg{4}$};

			\draw [->] (I) to node [swap] {$f$} (a3);
			\draw [->] (I) to node [] {$g$} (a2);

			\draw [->] ([yshift=-2]a3.east) to node [swap] {$f$} ([yshift=-2]a2.west);
			\draw [->] (a3) to node [swap] {$g$} (a1);

			\draw [->] (a2) to [loop right] node {$f$} (a2);
			\draw [->] ([yshift=2]a2.west) to node [swap] {$g$} ([yshift=2]a3.east);

			\draw [->] ([yshift=-2]a1.east) to node [swap] {$f$} ([yshift=-2]a4.west);
			\draw [->] (a1) to node [swap,xshift=-1,yshift=4] {$g$} (a2);

			\draw [->] (a4) to [loop right] node {$f$} (a4);
			\draw [->] ([yshift=2]a4.west) to node [swap] {$g$} ([yshift=2]a1.east);
		\end{tikzpicture}
		\caption{The interim automaton $\Ansq$ for Example~\ref{ex:1}, coming from
		Definition~\ref{def:automaton_nonsquare}. Boxed states are accepting.}
		\label{fig:ex1_Ansq}
	\end{figure}

	After performing the transformation described in the proof of
	Theorem~\ref{thm:irred_regular} and cutting out all unreachable states, we end
	up with the simple partial automaton $\Airr$ in Figure~\ref{fig:ex1_Airr}.
	This shows that the irreducible compositions of $f$ and $g$ are precisely
	those of the form $f^{(n)}$, $f^{(n)} \circ g$, $f^{(n)} \circ g^{(2)}$ and $f^{(n)} \circ g^{(2)} \circ f$ for $n \ge 0$.

	\begin{figure}[h]
		\centering
		\begin{tikzpicture}[auto]
			\node (S) at (0,0) [draw,circle] {$S$};

			\node (1) at ( 1.5, 0) [draw,circle] {$1$};
			\node (2) at ( 3.0, 0) [draw,circle] {$2$};
			\node (3) at ( 4.5, 0) [draw,circle] {$3$};

			\draw [->] (S) to [loop below] node {$f$} (S);
			\draw [->] (S) to node {$g$} (1);
			\draw [->] (1) to node {$g$} (2);
			\draw [->] (2) to node {$f$} (3);

		\end{tikzpicture}
		\caption{The automaton $\Airr$ accepting $\Lirr$ for Example~\ref{ex:1}. All states are accepting.}
		\label{fig:ex1_Airr}
	\end{figure}
\end{example}

Using the machinery we developed in the rest of the paper, we describe an
infinite set of primes of $\vF_q[x]$ having a finite regular structure.

\begin{theorem}\label{thm:disjoint}
Let $\vF_q$ be a finite field of characteristic different from $2$.
The set of irreducible polynomials having coefficients in $\vF_q$ which can be
written as a nonempty composition of degree $2$ polynomials can be partitioned
into a finite disjoint union $\bigsqcup_{a\in \vF_q} \cL_a$ in such a way that
each $\cL_a$ is in natural bijection with the words of a regular language $\cL$,
which is independent of $a$. In particular, the set of such irreducible
polynomials has a finite regular expression in terms of the elementary
operations $\cup, \cdot, {}^*$.
\end{theorem}
\begin{proof}
Let $D$ be the set of irreducible polynomials in $\vF_q[x]$ that can be written
as nonempty composition of degree $2$ polynomials. Let $\cS=\{x^2-b\}_{b\in \vF_q}$. By
Proposition~\ref{thm:freedom_max_min}, $C_\cS$ is isomorphic to $\cS^*$, so it is naturally
embedded in $\vF_q[x]$. Apply now Theorem~\ref{thm:irred_regular} to obtain  the
regular language of irreducible polynomials $\cI$ generated by $\cS$, and let
$\cL = \cI \setminus \{x\}$.
Let $\psi_a\colon \cL\longrightarrow \vF_q[x]$ be the shift map defined by
$f(x)\mapsto f(x+a)$.
Let $\cL_a=\psi_a(\cL)$.
It is easy to observe that for any polynomial $f\in D$, there exists $a\in
\vF_q$ such that $f(x-a)$ can be written as an element of $C_\cS$.
This shows that
\[D=\bigcup_{a\in \vF_q} \cL_a.\]
It remains to show that $\cL_a\cap \cL_b=\emptyset$ if $a\neq b$, the final
result will follow immediately. We argue by induction on the length of the words
in $\cL$ (i.e.\@ the degree of the polynomials). Let $a,b\in \vF_q$ with $a\neq b$
such that there exist two words $v,w\in \cL$ of minimal length $\ell$ such that
$\psi_a(v)=\psi_b(w)$. If $\ell=1$, this is clearly impossible, so let us assume
$\ell>1$. We can write $f(v'(x+a))=g(w'(x+b))$ for some $f,g\in \cS$ and $v',w'$
suffixes of $v$ and $w$ respectively. Therefore, for some $k,j\in \vF_q$ we have
\[v'(x+a)^2-w'(x+b)^2=(v'(x+a)-w'(x+b))(v'(x+a)+w'(x+b))=k-j.\]
Since $v',w'$ are monic and the characteristic of $\vF_q$ is different from $2$,
then the degree of the polynomial $(v'(x+a)+w'(x+b))$ is greater than or equal
to $2$. This forces both $k=j$ and $(v'(x+a)-w'(x+b))=0$, which contradicts the
minimality of $\ell$.
\end{proof}

\begin{example}\label{ex:2}
	For an example demonstrating Theorem~\ref{thm:disjoint}, take $q = 3$ and
	$\cS = \{f, g, h\}$ with $f = x^2$, $g = x^2 - 1$, $h = x^2 - 2$.
	From Proposition~\ref{thm:freedom_max_min}, $C_\cS$ is free and isomorphic to
	$\cS^*$.
	Applying the construction, we get the automaton shown in
	Figure~\ref{fig:ex2_Airr}.
	We see that the irreducible polynomials in $C_\cS$ are exactly $x$, $h$,
	$h\circ g\circ f^{(n)}$ for $n \ge 0$, and $h^{(2)}\circ k$ for $k \in C_\cS$ arbitrary (possibly the identity).
	Applying Theorem~\ref{thm:disjoint}, it follows that the set of irreducible polynomials in $\vF_3[x]$ that can be written as a nonempty composition of
	degree $2$ polynomials is precisely 
	\[\bigcup_{a\in \vF_3} \left(\{h(x+a)\} \cup \{h\circ g \circ f^{(n)}(x+a) : n \ge 0\} \cup \{h^{(2)}\circ k(x+a) : k \in C_\cS\}\right).\]

	\begin{figure}[h]
		\centering
		\begin{tikzpicture}[auto]
			\node (S) at (0,0) [draw,circle] {$S$};

			\node (1) at ( 1.5, 0) [draw,circle] {$1$};
			\node (2) at ( 3.0, 0) [draw,circle] {$2$};
			\node (3) at ( 1.5, -1.5) [draw,circle] {$3$};

			\draw [->] (S) to node {$h$} (1);
			\draw [->] (2) to [loop right] node {$f$} (2);
			\draw [->] (1) to node {$g$} (2);
			\draw [->] (1) to node {$h$} (3);
			\draw [->] (3) to [loop left] node {$f$} (3);
			\draw [->] (3) to [loop below] node {$g$} (3);
			\draw [->] (3) to [loop right] node {$h$} (3);

		\end{tikzpicture}
		\caption{The automaton $\Airr$ accepting $\Lirr$ for Example~\ref{ex:2}. All states are accepting.}
		\label{fig:ex2_Airr}
	\end{figure}
\end{example}

Let us now describe two implications of our result in the case in which $q$ is small compared with $n$. Recall that the \emph{iterated logarithm} of a positive real number $x$, denoted by $\log^*x$, is the number of times the logarithm function must be iteratively applied before the result is less than or equal to $1$.
\begin{corollary}
For fixed $q$, we can list all monic irreducible polynomials of degree $2^n$ which are compositions of degree $2$ polynomials with complexity $O(q^n  2^n  n  8^{\log^*(2^n)})$, where the implied constant depends only on $q$.
\end{corollary}
\begin{proof}
First, we choose $\cS=\{x^2-a: \;a\in \mathbb F_q\}$ and write down the automaton $\mathcal M$ given by Theorem \ref{thm:irred_regular}. The complexity of this step is $O(1)$, where the implied constant clearly depends only on $q$.
Since the distinguished set is maximal, we can identify the polynomials $C_\cS$ and the words in $\cS^*$.
Now the construction is recursive: let $\cL_0(m)$ be the set of words of $\cS^*$ of length $m$ (so polynomials of degree $2^m$). For any word $g\in \cL_0(m)$ and any $f\in \cS$, check if the word $gf$ is accepted by the automaton $\mathcal M$: if it is, $gf$ is an element in $\cL_0(m+1)$. This check takes constant time for fixed $q$ if for each element of $\cL_0(m)$ we also store the state in which the automaton ends after reading it.
As we observed in Remark~\ref{rem:prefixclosed}, the language accepted by $\mathcal M$ is prefix closed, so  all the elements of $\cL_0(m+1)$ can be constructed in this way. Then, constructing $\cL_0(m+1)$ from $\cL_0(m)$ costs at most $O(q^{m+1})$.

It is elementary now to observe that $\cL_0(1)$ can be easily constructed by selecting the irreducible degree $2$ polynomials (again $O(1)$ operations) and therefore that $\cL_0(n)$ can be constructed in time at most $O(q^n)$.
Using the proof of Theorem~\ref{thm:disjoint} directly, we know that the set $D(n)$ of all irreducible polynomials of degree $2^n$ which are compositions of degree $2$ polynomials can be written as
\[D(n)=\bigsqcup_{a\in \vF_q} \cL_a(n),\]
where $\cL_a(n)=\{g(x+a): \; g\in \cL_0(n)\}$.

In order to represent the elements of $D(n)$ by coefficients, we now need to evaluate the elements of $\cL_0(n)$.
For a single such element, if we evaluate it from inside out, this means squaring polynomials of degree $d = 1, 2, \ldots, 2^{n-1}$ and subtracting a constant each time.
According to~\citep{bib:harvey2017faster}, such a squaring can be done in time $O(d \log(d) 8^{\log^*(d)})$, the total time for each element hence being $O(2^n n 8^{\log^*(2^n)})$.
\end{proof}

\begin{remark}
The reader should notice that this is much quicker than listing such polynomials by using an irreducibility test. This would have in fact complexity $O(q^n  2^n  n  8^{\log^*(2^n)}  I(2^n))$, where $I(2^n)$ is the cost of the chosen irreducibility test for a polynomial of degree $2^n$.  Again, the implied constant depends on the time of constructing the automaton, which just depends on the parameter $q$.
\end{remark}

Another interesting fact is that (again for fixed $q$) we have an efficient deterministic algorithm to test irreducibility for polynomials which are a composition of degree two polynomials.

\begin{corollary}\label{cor:irred_test_decomp}
	Let $f$ be a polynomial of degree $2^n$ which is known to be a composition of
	monic degree two polynomials.
	Then $f$ can be tested for irreducibility in time $O(2^nn^2 8^{\log^*(2^n)})$,  where the implied
	constant depends only on $q$.
\end{corollary}
\begin{proof}
	We need to write $f$ in the form $f = g_1 \circ g_2 \circ \cdots \circ g_n
	\circ \ell$, where $g_i = x^2 - a_i$ and $\ell = x - b$, with $a_i$ and $b$ in
	$\vF_q$.

	For this, suppose we have $F = G \circ H$ with $G = x^2 - a$, where $a \in
	\vF_q$, and $H$ is a polynomial of degree $d \ge 1$.
	Assume $F$ is known, and we seek $G$ and $H$.
	We apply the algorithm \texttt{Univariate decomposition}
	from~\cite[Section~2]{bib:von1990functional} with $r = 2$.
	This gives us the unique $\tilde{G}$ of degree $2$ and $\tilde{H}$ of degree
	$d$ with $\tilde{H}(0) = 0$ such that $F = \tilde{G} \circ \tilde{H}$.
	Clearly, setting $c = H(0)$, we have $\tilde{H} = H - c$ and $\tilde{G} = G(x
	+ c) = x^2 + 2cx + c^2 - a$.
	From this, it is easy to recover $c$, $a$, $G$ and $H$.
	The algorithm \texttt{Univariate decomposition} takes time $O(d \log(d)^2
	8^{\log^*(d)})$, again using the multiplication algorithm
	from~\citep{bib:harvey2017faster}.

	Applying this repeatedly to the $f$ from the theorem will find the
	decomposition $f = g_1 \circ g_2 \circ \cdots \circ g_n \circ \ell$ in $O(2^n
	n^2 8^{\log^*(2^n)})$ time, which can then be checked for irreducibility with
	the automaton in time $O(n)$. Again, for fixed $q$, constructing the automaton has constant complexity, independently of the degree of the polynomial we are testing.
\end{proof}

\begin{remark}
For fixed $q$, testing irreducibility for a polynomial of degree $2^n$ using for example Rabin's test \cite{bib:rabintest} with fast polynomial operations costs $O(n4^{n})$. 
\end{remark}

\begin{remark}
As we already mentioned, the whole point is that our algorithm is very efficient
in the regime of small fixed $q$ and large $n$.
Let us nonetheless have a quick look at the complexity with regards to $q$.
If we follow the algorithm as described in
Corollary~\ref{cor:irred_test_decomp} directly, the complexity appears to be
exponential in $q$.
In particular, we can expect the subset construction step  to take
exponential time and space.

Fortunately, it is not in fact necessary to construct the entire automaton to
execute the above algorithm.
Instead, we can solely construct the interim automaton $\mathcal N$ from
Theorem~\ref{thm:irred_regular}.
This takes $O(q^2)$ field operations, and has to be done only once for each $q$.
Then, given the decomposition of the polynomial into $f_1 \circ \cdots\circ f_n$ we can run a word $f_1 \cdots f_n$ through $\mathcal N$ directly as
follows:
define $S_0$ as the set of accepting states of $\mathcal N$.
Then, for $i$ from $1$ to $n$, let $S_i$ be the set of all states $t$ such that
there is a $s \in S_{i-1}$ and a transition $t \xrightarrow{f_i} s$ in $\mathcal
N$.
If at some point $S_i$ does not contain the initial state $\init$, we reject the
word.
Otherwise, we accept.

This method mirrors the reversal and subset construction from
Theorem~\ref{thm:irred_regular}, except that only the parts that are actually
used are computed.
The complexity of this algorithm is easy to determine:
when computing $S_i$ from $S_{i-1}$, we can iterate over all states $t$ of the
automaton and check whether the unique outgoing transition labelled $f_i$ ends
in $S_{i-1}$.
Hence, each step takes only $O(q)$ field operations, for a total cost of $O(q^2
+ nq)$.
Since for small $q$ the quantity $2^n	n^2 8^{\log^*(2^n)}$ dominates $q^2+nq$, the
complexity in terms of $\mathbb F_q$-operations remains unchanged.
\end{remark}

\section{Irreducible compositions over local fields}

In this final section, we will show how the results of the previous sections can be lifted, under some additional hypothesis, to polynomials over local fields. Let $K$ be a non-archimedean local field with finite residue field $\vF_q$ of odd characteristic. Let $\cO_K$ be its ring of integers and $\varpi$ be a uniformizer.  We will denote by $\widetilde{\cdot}$ the reduction map $\cO_K[x]\to \vF_q[x]$. Let us start by recalling the following lemma, which we state in a weaker form, sufficient for our purposes.
\begin{lemma}\label{rafe_lemma}
Let $L$ be any field and let $f,g\in L[x]$ be monic polynomials with $f=(x-a_f)^2-b_f$ for some $a_f,b_f\in L$. Then we have:
\[\disc(g\circ f)=\pm \disc(g)^{2}\cdot 4^{\deg g}\cdot g(-b_f).\]
\end{lemma}
\begin{proof}
 See \cite[Lemma 2.6]{bib:jones}.
\end{proof}

\begin{theorem}\label{local_thm}
 Let $f_1,\ldots,f_k\in \cO_K[x]$ be monic polynomials of degree 2 such that $\varpi \nmid \disc(f_1)$. Then $f_1\circ\cdots\circ f_k$ is irreducible in $K[x]$ if and only if $\widetilde{f_1}\circ\cdots\circ \widetilde{f_k}$ is irreducible in $\vF_q[x]$.
\end{theorem}
 \begin{proof}
    
   One direction is obvious, so let us assume that $f_1\circ\cdots\circ f_k$ is irreducible. For every $i=1,\ldots,k$, let $f_i=(x-a_i)^2-b_i$ for some $a_i,b_i\in \cO_K$. By Proposition~\ref{thm:irred_long}, we need to show that the following elements are not squares:
   \begin{itemize}
		\item $c_1\coloneqq\widetilde{b_1}$
		\item $c_2\coloneqq\widetilde{f_1}(\widetilde{-b_2})$
		\item[] $\vdots$
		\item $c_k\coloneqq(\widetilde{f_1} \circ \cdots \circ \widetilde{f_{k-1}})(\widetilde{-b_k})$.
	\end{itemize}
    First, suppose that $c_t=0$ for some $t\in \{1,\ldots,k\}$. This implies that $\widetilde{f_1}$ has a root, and since by hypothesis the discriminant of $\widetilde{f_1}$ is non-zero, by Hensel's lemma we can lift such a root to a root of $f_1$. But then $f_1\circ\cdots\circ f_k$ is clearly reducible, which is a contradiction. Thus we can assume that $c_i\neq 0$ for all $i\in\{1,\ldots,k\}$. Now let $t\in \{1,\ldots,k\}$ be such that $c_t$ is a non-zero square. By Proposition~\ref{thm:irred_long}, this implies that $\widetilde{f_1}\circ\cdots\circ \widetilde{f_t}$ is reducible. On the other hand, applying Lemma~\ref{rafe_lemma} recursively and using the definition of the $c_i$'s we get that:
	\[\disc(\widetilde{f_1}\circ\cdots\circ \widetilde{f_t})=u \cdot\prod_{i=1}^{t}c_i^{2^{t-i}}\neq 0,\]
    where $u$ is an appropriate power of 2 (up to sign). This proves that $\varpi \nmid \disc(f_1\circ\cdots\circ f_t)$ and since $f_1\circ\cdots\circ f_t$ is irreducible by hypothesis, it defines an unramified extension of $K$. It follows that $\widetilde{f_1}\circ\cdots\circ \widetilde{f_t}$ is irreducible (see for example~\cite[Chapter~7]{bib:cassels}), giving a contradiction.
 \end{proof}
 It is clear that the hypothesis that $\varpi\nmid \disc(f_1)$ is necessary for the claim to hold, since for example $x^2-\varpi$ is irreducible in $K[x]$, while its reduction is reducible in $\vF_q[x]$.
 
 Given a finite set $\cS\subseteq \cO_K[x]$ of monic polynomials of degree two with unitary discriminant, Theorem \ref{local_thm} shows that irreducible compositions of the elements of $\cS$ correspond bijectively to irreducible compositions of the elements of $\widetilde{\cS}\subseteq \vF_q[x]$. Therefore, if we consider $\cS$ as an alphabet and $\Lirr$ is the language of irreducible compositions of the elements of $\cS$, we deduce immediately the following corollary.
 \begin{corollary}
  The language $\Lirr$ is regular.
 \end{corollary}
 \begin{proof}
  It is enough to apply Theorem~\ref{thm:irred_regular} to the language of irreducible compositions of the elements of $\widetilde{\cS}$.   
 \end{proof}
 The above corollary essentially states that the theory we developed in the rest of the paper lifts entirely to local fields, at least in the case in which the elements in $\cS$ have unitary discriminant. 
It would be interesting to understand what happens when this condition is not satisfied.
\section{Further research}
One of the natural questions arising from the results in the present paper is whether Theorem~\ref{thm:disjoint} can be generalised to higher degree polynomials. In fact any lift of such results to polynomials of degree three or more would be of great interest, in particular because the necessary and sufficient criterion by Boston and Jones~\citep{bib:jones2012settled} (and the subsequent results on the subject such as~\citep{bib:ahmadi2009note,bib:ahmadi2012stable,ferraguti2016sets,bib:gomez2010estimate,bib:heathbrownmicheli}) only exists in degree two. 
In the context of local fields, another interesting issue arising from Theorem \ref{local_thm} of this paper is the following: how can one include singular polynomials in the generating set $\cS$? In fact, the condition on the discriminant seems to be essential.
Another question arising from these results is whether it is possible to explicitly compute the generating function of the language of irreducible compositions just in terms of the coefficient of the generating polynomials. It is indeed possible to compute such a function just in terms of the obtained automata, but this would drop most of the available information. In particular, it seems that the key ingredient to address this issue is to understand the structure of the finite submonoid of maps from $\vF_q$ to $\vF_q$ which can be written as composition of degree two polynomials. 
More generally, many of the questions and constructions which were related just to the compositions of a single polynomial, now seem to naturally arise in this more general context, were the rigidity of finite automata theory provides assistance.

\section{Acknowledgements}

The first author was supported by Swiss National Science Foundation grant number 168459. The Second Author is thankful to Swiss National Science Foundation grant number 161757 and 171248.

The authors are thankful to the anonymous referee for his suggestions, which greatly improved the content and the readability of the paper. We are expecially thankful for suggesting a reference which boiled down the complexity in Corollary \ref{cor:irred_test_decomp} from $O(4^n)$ to $O(2^n	n^2 8^{\log^*(2^n)})$.

\bibliographystyle{plain}
\bibliography{biblio}

\end{document}